\documentclass[preprint,review]{elsarticle}
\usepackage{amssymb}
\usepackage{amsthm,amsmath}
\usepackage{filecontents, color, url}
\usepackage{lineno}
\usepackage{chapterbib}

\journal{Linear and Multilinear Algebra}

\newtheorem{thm}{Theorem}[section]

\newtheorem{lem}[thm]{Lemma}
\newtheorem{cor}[thm]{Corollary}

\theoremstyle{definition}

\newtheorem{ex}[thm]{Example}

\theoremstyle{remark}

\newcommand{\R}{\mathbb{R}}  
\newcommand{\N}{\mathbb{N}}  
\newcommand{\C}{\mathbb{C}} 
\newcommand{\HM}{\mathcal{H}} 
\newcommand{\s}{\{s_k\}_{k\in \N_0}} 
\newcommand{\st}{\{\tilde{s}_k\}_{k\in \N_0}} 
\newcommand{\subs}{\{\tilde{s_k}\}_{k\in \N_0}} 
\newcommand{\ts}{\{t_k\}_{k\in \N_0}} 
\newcommand{\intR}{\int_{-\infty}^{\infty}} 

\begin{document}

\begin{frontmatter}

\author[add1]{Saroj Aryal}
\ead{saroj.aryal@msubillings.edu}
\author[add2]{Hayoung Choi\corref{cor}}
\ead{hchoi2@uwyo.edu}
\author[add2]{Farhad Jafari}
\ead{fjafari@uwyo.edu}

\cortext[cor]{Corresponding author}

\address[add1]{Department of Mathematics, Montana State University, Billings, MT 59101, USA}
\address[add2]{Department of Mathematics, University of Wyoming, Laramie, WY 82071, USA}

\title{Hamburger moment sequences and their moment subsequences}

\begin{abstract}
In this paper a connection between Hamburger moment sequences and their moment subsequences is given and the determinacy of these problems are related.
\end{abstract}

\end{frontmatter}

\section{Introduction and Main Results}

Given a sequence of real numbers $\s$ a necessary and sufficient condition that there exists a non-decreasing function $\sigma $ on $\R$
such that
\begin{equation}\label{Hamburger_moment}
s_{k}=\int_{-\infty}^{\infty} x^k \mathrm{d}\sigma(x) \quad \textup{for all } k\in \N_{0}
\end{equation}
is that the quadratic forms
\begin{equation*}
\sum_{i,j=0}^n x_i x_j s_{i+j} \geq 0 \quad \forall n\in
\N_0,~ x_0,\ldots,x_n \in \R.
\end{equation*}
This is equivalent to the Hankel matrices
\begin{equation}
\HM_n:=(s_{i+j})_{i,j=0}^{n}
\end{equation}
being positive semidefinite for all $n \in \N_{0}$.

The sequences with representation \eqref{Hamburger_moment} are called \emph{(Hamburger) moment sequences} or \emph{positive sequences}. The function $\sigma$ is called a \emph{moment solution} of the sequence. Positive sequences consist of two mutually disjoint sequences. A sequence is called \emph{positive definite} if the Hankel matrix $\HM_n$ is positive definite for all $n\in \N_0$; otherwise it is called \emph{positive semidefinite}.
It is worthwhile to point out that there is a significant difference between definite and semidefinite character of moment sequences in their integral representations \cite{BS15, book:widder}.

Two nondecreasing functions are not considered to be distinct if their difference is a constant at all the points of continuity of the difference. 
A Hamburger moment sequence with no more than one distinct nondecreasing function $\sigma$ in \eqref{Hamburger_moment} is called \emph{determinate}, and \textit{indeterminate} otherwise.
Berg, Chen and Ismail \cite{BCI02} proved that a moment sequence is determinate if and only if $\lambda_n \rightarrow 0$ for $n \rightarrow 0$, where
$\lambda_n$ is the smallest eigenvalue of $\HM_n$.
Moreover, for the indeterminate case a positive lower bound for the smallest eigenvalue of the Hankel matrices was explicitly found.

A \emph{partial sequence} is a sequence in which some terms are specified, while the remaining terms are unspecified and may be treated as free real variables.
A \emph{partial positive (semi)definite sequence} is a partial sequence if each of the fully specified principal submatrices of its Hankel matrix $\HM_n$ is positive (semi)definite for all $n \in \N_{0}$. Since every principal submatrix of a positive definite matrix is positive definite, it is trivial that a partial sequence has a positive completion only if it is a partial positive sequence.
A \emph{Hamburger moment completion} (or a \emph{positive completion}) of a partial sequence is a specific choice of values for the unspecified terms resulting in a positive (semi)definite sequence. Note that a partial positive definite sequence can have either a positive definite completion or a positive semidefinite completion, and possibly both.
The \emph{Hamburger moment completion problem} asks whether a given partial sequence has a Hamburger moment completion. 
A \emph{pattern} of a partial sequence is the set of positions of the specified entries. Denote the pattern of a partial sequence by the set of positive integers
\begin{equation*}
P=\{k \in \N_0 ~:~ s_k \text{ is specified}\}.
\end{equation*}
We say that a pattern $P$ is \emph{positive (semi)definite completable} if every partial positive (semi)definite sequence with pattern $P$ has a positive (semi)definite completion.

There are two main methods for dealing with partial moment sequences: (i) Perturbation or modification, and (ii) Completions. Perturbation or modification investigates the stability of positivity of a moment sequence to perturbations by arbitrary sequences, thus informing about the range of possible values allowable for the missing entries. Completion is based on reconstruction from fully specified subsequences of a partial sequence while retaining positivity. Perturbation of a moment sequence is closely related to the orthogonal polynomials the moment sequence generates. Gautschi \cite{gaustchi} studies sensitivity of orthogonal polynomials to perturbations. Readers are advised to refer to \cite{beckermann} and \cite{sunmikim}  for examples of modified moments. Choi and Jafari \cite{CJ15} give solutions to the Hamburger moment completion problem. In this paper, a comprehensive study of completable partial positive sequences and their patterns is given and perturbations of moment sequences by arbitrary sequences is studied.

In Section 2 of this paper, a series of results are given that characterize the moment subsequences of Hamburger moment sequences and give a description of completable patterns from partial sequences.  Some of the results and examples on subsequences have appeared in the monograph of Berg, Christensen and Ressel \cite{BCR} and the extensive paper of Stochel and Szafraniec \cite{SS} and elsewhere. While not new, alternative elementary proofs are presented here to set the stage to discuss the subsequence problem and compeletable patterns. In Section 3, a relationship between moment and submoment sequences is derived and the determinacy of the two problems are compared. Finally, in Section 4, stability of the moment problem with respect to perturbation by an arbitrary sequence is discussed and a criterion for such perturbations is given.

\section{Moment subsequences and completions}
Let $\s$ be a moment sequence. If its subsequence $\st$ is also a moment sequence, it is called a \emph{submoment sequence}.

\begin{ex}\label{ex:Hilbert}
Consider the sequence $\s$ given by
\begin{equation}\label{ex1}
s_k = \frac{1}{k+1}.
\end{equation}
Since the Hankel matrices $\HM_n$ are positive definite for all $n\in \N_0$, the sequence is a Hamburger moment sequence. 
Indeed, each of the Hankel matrices $\HM_n$ is a Hilbert matrix, which is totally positive (all of its minors are positive). Thus there exists a non-decreasing function ${\sigma} $ on $\R$ such that
\begin{equation}\label{ex1:submoment}
\frac{1}{k+1}=\int_{-\infty}^{\infty} x^k \mathrm{d}{\sigma}(x) \quad \textup{for all } k\in \N_0.
\end{equation}
Let $\epsilon>0$ and choose $m\in \N_0$ such that $ s_{2m}< \epsilon$. Suppose that only the term $s_{2m}$ is missing and the remaining terms are specified.
Setting $s_{2m}=0$ the Hankel matrix $\HM_{m}$ is not positive semidefinite. Thus, arbitrarily small perturbations of a positive sequence ejects one from the cone of positive sequences. We will return to perturbation of positive sequences in Section 4.
\end{ex}
For positive semidefinite completable patterns, arithmetic progression patterns play a crucial role. In \cite{CJ15} Choi and Jafari show that arithmetic progression patterns guarantee there exists completions.
\begin{thm}\label{thm:mainresult}
If the pattern $P = d\N_0+\ell_0$ for some $d\in \N$ and $\ell_0 \in 2\N_{0}$, then $P$ is positive semidefinite completable. 
\end{thm}
The following corollary shows that if $\ell_0 = 0 $, then the partial positive pattern $P$ is actually positive definite completable. 

\begin{cor}\label{cor:mainresult2}
If the pattern $P = d\N_0$ for some $d\in \N_0$, then $P$ is positive definite completable.
\end{cor}
However, if the pattern $P = d\N+\ell_0$ for $d\in \N$ and $\ell_0 \in 2\N$, then the pattern $P$ is not positive definite completable.

\begin{ex}
The subsequence $\st$ of the sequence $\s$ in Example \ref{ex:Hilbert} given by 
\begin{equation}\label{notsubmoment}
\tilde{s}_{k}=\dfrac{1}{(k+1)!}
\end{equation}
is not a positive sequence since the determinant of its second Hankel matrix is 
$$\det{\HM_2}=\left|\begin{matrix}1 & \frac{1}{2} \\[1em]
\frac{1}{2} & \frac{1}{6}
\end{matrix}
\right|<0$$
and hence has no solution to its moment problem. This example shows that an arbitrary subsequence of a positive sequence is not necessarily positive. Thus we need to look for ideas to appropriately remove terms of the sequence while preserving positivity.
\end{ex}

\begin{thm}\label{mean}
Let $\s$ be a positive sequence. The subsequence $\st$ given by $\tilde s_k=s_{k+\ell_k}$ is positive if $\ell_k=kd+\ell_0$ for all $d\in\N_0$ and $\ell_0\in 2\N_0$.
\begin{proof}
By Hamburger's theorem, since $ \s $ is a positive sequence, there exists a non-decreasing function $ \sigma $ on $\R$ that has $ \s $ as its moment sequence. Suppose $\ell_k=kd+\ell_0$ where $d\in\N_0$, and $\ell_0\in 2\N_0$. Then for all real numbers $x_0,x_1,x_2,\cdots, x_m$, we have
\begin{align*}
\sum_{i,j=0}^m x_i x_j \tilde s_{i+j}
&= \sum_{i,j=0}^m x_i x_j s_{i+j+(i+j)d+\ell_0}\\
&= \sum_{i,j=0}^m x_i x_j \intR x^{(1+d)i+(1+d)j+\ell_0} d\sigma(x)\\
&= \intR \left[x^{\frac{\ell_0}{2}}\sum_{i=0}^m x_i x^{(1+d)i}\right]^2d\sigma(x)\geq 0.
\end{align*}
\end{proof}
\end{thm}
The above shows that any subsequence of a positive sequence is positive if they are extracted in a certain periodic manner.

\begin{ex}
Now consider the following partial sequence
\begin{equation}\label{intro:ex3}
 1,~ ?,~\frac{1}{2},~ ?, ~\frac{1}{3},~?, ~\frac{1}{4},~?, ~\frac{1}{5},~?, ~\frac{1}{6},~?, \cdots.
\end{equation}
Since the specified subsequence $\{ 1/(k+1)\}_{k=0}^{\infty}$ is positive definite, there exists a non-decreasing function ${\tilde{\sigma}}$ on $\R$ such that
\begin{equation*}
\frac{1}{k+1}=\int_{-\infty}^{\infty} x^k \mathrm{d}{\tilde{\sigma}}(x) \quad \text{for all } k\in \N_{0}.
\end{equation*}
By Corollary \ref{cor:mainresult2} the sequence \eqref{intro:ex3} has a positive definite completion. That is, there exists a non-decreasing function $ \sigma $ on $\R$ such that
\begin{equation*}
\frac{1}{k+1}=\int_{-\infty}^{\infty} x^{2k} \mathrm{d} \sigma (x) \quad
\textup{for all } k\in \N_0.
\end{equation*}
\end{ex}
The primary goal in this paper is to investigate connections between solutions of  a moment sequence and solutions of its submoment sequences.

As was seen in an above example, an arbitrary subsequence of a positive sequence is not necessarily positive. We will develop some methods to extract positive subsequences from a positive sequence.

\begin{thm}\label{positiveseq1}
Let $\{f_k\}_{k\in \N_0}$ be a sequence of functions on $\R$ such that 
\begin{equation}\label{cond1}
f_i(x)f_j(x)=f_{i+j}(x).
\end{equation}
Let $\sigma$ be a non-decreasing function on $\R$ such that
\begin{equation}
s_{k}:=\int_{-\infty}^{\infty} f_k(x)d\sigma(x)<\infty  \textit{ for all } k \in \N_0.
\end{equation}
Then the sequence $\s$ is positive.
Furthermore, condition \eqref{cond1} holds if and only if
\begin{equation}\label{geometricform}
f_k(x)= (f_1(x))^k \quad \text{for all } k\in \N_0.
\end{equation}
\end{thm}
\begin{proof}
For any finite set of real numbers $x_0,x_1,\cdots,x_m$, we have 
\begin{align*}
\sum_{i,j=0}^m x_i x_j s_{i+j}
&= \sum_{i,j=0}^m x_i x_j\intR f_{i+j}(x) d\sigma(x)\\
&= \sum_{i,j=0}^m x_i x_j\intR f_i(x)f_j(x) d\sigma(x)\\
&= \intR \sum_{i,j=0}^m x_i x_j f_i(x)f_j(x) d\sigma(x)\\
&= \intR \left[\sum_{k=0}^m x_k f_k(x)\right]^2 d\sigma(x) \geq 0.
\end{align*}
Clearly \eqref{geometricform} holds for $k=0$ and $k=1$. Assume $f_n(x)=(f_1(x))^n$ holds. Then $f_{n+1}(x)=f_n(x)f_1(x)=(f_1(x))^{n} f_1(x)=(f_1(x))^{n+1}$, inductively. Proof of the converse is trivial.
\end{proof}

Note that Theorem \ref{positiveseq1}  allows many constructions of positive sequences. For example,  letting $f_k(x)=x^k$ we obtain precisely the classical power moment sequence for any non-decreasing function $\sigma $ such that the integrals are finite. Similarly, letting $f_k(x)=a^{kx}$ for a nonzero constant $a\in \R$ or $(\phi(x))^k$ for any $\sigma-$measurable function $\phi$ gives us more positive sequences corresponding to a non-decreasing functions $\sigma$. Theorem \ref{positiveseq1} says that all positive sequences are obtained in this way.

Now using Theorem \ref{positiveseq1}, we can construct positive subsequences from such moment sequence. Let $ \sigma $ be a non-decreasing function on $\R$ and $\s$ be a sequence defined as 
$$s_{k}=\int_{-\infty}^{\infty}(\phi(x))^k d\sigma(x)$$ 
for some $\sigma$-integrable function $\phi(x)$ for which all the above integrals are finite. Then for a fixed $\ell \in \N_0$ the sequence 
$$s_{k_\ell}
:= \int_{-\infty}^{\infty} \left((\phi(x))^\ell\right)^{k} d\sigma(x)
= \int_{-\infty}^{\infty} (\phi(x))^{k\ell} d\sigma(x)
= s_{k\ell}$$ is a positive subsequence of $\s$.
Given a positive sequence $\s$, the problem of identifying all the positive subsequences requires finding all the sequences $\{\ell_k\}\subseteq \N_0$ such that the sequence given by $$\tilde{s}_k=s_{k+\ell_k}$$is a positive sequence.

\begin{thm}\label{geometric}
The sequence $\{a^{\ell_k}\}_{k\in \N_0}$ is a positive sequence for each $a\in \R$ if and only if
$\ell_k=kd+\ell_0$ for some $d\in\N_0$ and $\ell_0\in 2\N_0$.
\begin{proof}
Suppose $\ell_k=kd+\ell_0$ for some $d\in\N_0$ and $\ell_0\in 2\N_0$. Fix $a\in \R$.
Define $\delta :\R \longrightarrow \R$ as a non-decreasing function with only one point of increase $a^d$
such that
\begin{equation*}
a^{\ell_0} = \delta(a^d + 0) - \delta(a^d - 0).
\end{equation*}
Then, by the definition of the Stieltjes integral
\begin{equation*}
a^{\ell_k}=a^{dk+\ell_0}=\intR x^k d\delta(x) \quad \textit{for all } k\in\N_0.
\end{equation*}
Thus, $\{a^{\ell_k}\}_{k\in \N_0}$ is a positive sequence.
Suppose $\{a^{\ell_k}\}_{k\in \N_0}$ is a positive sequence for each $a\in \R$. Then by Hamburger's Theorem the corresponding Hankel matrix $\HM_n$ is positive semidefinite for all 
$n\in \N_0$. Since all its principal minors are nonnegative, $a^{\ell_0} \geq 0$ for all $ a\in \R $, thus $\ell_0 \in 2\N_0$. Since the following $2 \times 2$ principal submatrix of $\HM_{\ell_i}$
\begin{equation*}
\begin{bmatrix}
a^{\ell_i} & a^{\ell_{i+1}}\\
a^{\ell_{i+1}} & a^{\ell_{i+2}}
\end{bmatrix}
\end{equation*}
must have nonnegative determinant, 
$$a^{\ell_i} a^{\ell_{i+2}} - a^{2\ell_{i+1}} \geq 0$$ for all $\ell_i \in 2\N_0$.
Since $a\in\R$ is arbitrary, 
\begin{equation}\label{arithmetic1}
\ell_i + \ell_{i+2} = 2\ell_{i+1}  \quad \text{for all }\ell_i \in 2\N_0.
\end{equation}
Similarly, using the $2 \times 2$ principal submatrix of $\HM_{\ell_i}$
\begin{equation*}
\begin{bmatrix}
a^{\ell_i} & a^{\ell_{i+2}}\\
a^{\ell_{i+2}} & a^{\ell_{i+4}}
\end{bmatrix}
\end{equation*}
\begin{equation}\label{arithmetic2}
\ell_i + \ell_{i+4} = 2\ell_{i+2} \quad \text{for all }\ell_i \in 2\N_0.
\end{equation}
By \eqref{arithmetic1} and \eqref{arithmetic2} it follows that
the sequence $\{ \ell_i,~\ell_{i+1},~\ell_{i+2},~\ell_{i+3},~\ell_{i+4} \}$
is an arithmetic progression for each $\ell_i \in 2\N_0$ and $i\in \N_0$. Thus,
the sequence $\{\ell_i\}_{i\in \N_0}$ has the desired form.
\end{proof}
\end{thm}

For later use, we note that while showing Theorems \ref{mean} and \ref{geometric}, we also have proved the following. 
\begin{cor}
Let $\s$ be a positive sequence. Then the subsequence 
$\{\tilde s_k\}_{k\in \N_0}$ given by
$\tilde s_k=s_{k+\ell_k}$ is a positive sequence if $\{a^{\ell_k}\}_{k\in \N_0}$ is a positive sequence for each $a \in \R$.
\end{cor}

\section{Relationship between moments and their submoments}

Assume that $\s$ is a moment sequence and indeterminate.
Then the Hankel matrices $\HM_n$ corresponding to $\s$ are positive semidefinite for all $n \in \N_0$ and there exist distinct non-decreasing functions $\sigma_1$ and $\sigma_2$ such that
\begin{equation*}
s_{k}=\int_{-\infty}^{\infty} x^k \mathrm{d}\sigma_1 (x)\quad \text{and} \quad  s_{k}=\int_{-\infty}^{\infty} x^k \mathrm{d} \sigma_2 (x) \quad 
\text{for all } k\in \N_{0},
\end{equation*}
respectively.
Let $d \in \N$ and $\ell_0 \in 2\N_0$. 
Let $\subs$ be the subsequence of $\s$ given by 
\begin{equation*}
\tilde{s}_k:=s_{kd+\ell_0} \quad \text{for all } k\in \N_0.
\end{equation*}
Then it is trivial that
\begin{equation*}
\int_{-\infty}^{\infty} x^{kd+\ell_0} \mathrm{d}\sigma_1 (x)
=  \int_{-\infty}^{\infty} x^{kd+\ell_0} \mathrm{d} \sigma_2(x)
\quad \text{for all } k\in \N_{0},
\end{equation*}
respectively.
Since every principal submatrix of a positive (semi)definite matrix is positive (semi)definite, $\subs$ is positive (semi)definite. This implies that there is a nondecreasing function $\sigma$ such that
\begin{equation}
\tilde{s}_k = \int_{-\infty}^{\infty} x^{k} \mathrm{d}\sigma (x)
\quad \text{for all } k\in \N_{0}.
\end{equation}
That is,
\begin{equation*}
\int_{-\infty}^{\infty} x^{k} \mathrm{d}\sigma (x)
= \int_{-\infty}^{\infty} x^{kd+l_0} \mathrm{d} \sigma_1 (x) 
= \int_{-\infty}^{\infty} x^{kd+l_0} \mathrm{d} \sigma_2 (x) 
\end{equation*}
for all $k\in \N_{0}$.
Note that even though there are two distinct non-decreasing functions $\sigma_1$ and $\sigma_2$, it is not clear whether such a non-decreasing function $\sigma$ is unique or not.

As mentioned earlier, the result characterizing the determinacy by the eigenvalue of the Hankel matrices exists as follows \cite{BCI02}.
\begin{thm}\label{Berg:determinacy}
Let $\s$ be a moment sequence and $ \lambda_n $ be the smallest eigenvalue of its Hankel matrix $ \HM_n $. Then
$\s$ is determinate if and only if $\lambda_n \rightarrow 0$ as $n \rightarrow \infty$.
Furthermore, in the indeterminate case the positive lower bound is explicitly found.
\end{thm}

Using the preceding theorem one can show the following.
\begin{thm}
Let $\s$ be a positive sequence and $\subs$ be its subsequence given by $\tilde{s}_k=s_{kd+\ell_0}$ for $d\in \N$ and $\ell_0 \in 2\N_0$. If $\s$ is indeterminate, then $\subs$ is indeterminate.
\begin{proof}
Let $\HM_n$ and $\tilde{\HM}_n$ be $n \times n$ Hankel matrices corresponding to $\s$ and $\subs$, respectively. 
Since $\s$ is indeterminate, by Theorem \ref{Berg:determinacy} the smallest eigenvalues of $ \HM_n $ have a positive lower bound $C>0$, 
Let $\lambda$ and $\tilde{\lambda}_n$ be the smallest eigenvalues of $\HM_n$ and $\tilde{\HM}_n$, respectively.
Since $\tilde{\HM}_n$ is a principal submatrix of $\HM_n$, by the Cauchy interlacing theorem
$\tilde{\lambda}_n \geq \lambda_n \geq C >0$ for all $n\in \N_0$. So $\tilde{\lambda}_n$ has the positive lower bound $C$ for all $n\in \N_0$. Thus $\subs$ is indeterminate.
\end{proof}
\end{thm}
It is easy to note that this also implies that if $ \subs $ is determinate, then $ \s $ is also determinate. Thus, determinacy of the $\s$ may be deduced from a sparse subsequence of the original sequence.

Given two positive sequences $\s$ and $\st$ with moment solutions $\sigma$ and $\tilde\sigma$ respectively, if $\sigma(x)=\tilde\sigma(x)$, then it is clear that $ s_k = \tilde s_k $ for all $ k\in\N_0 $. Therefore the moment solutions to any two distinct positive sequences can never be equal. Uniqueness in this correspondence raises the question of how $\sigma(x)$ and $\tilde\sigma(x)$ compare.
\begin{ex}
Consider a positive sequence $\s$ with its moment solution $\sigma$. If some of the terms from $\s$ are missing then we obtain a subsequence which, if positive, gives another moment solution. For an appropriately chosen $\{\ell_k\}\subseteq \N_0$, assume that $\tilde s_k=s_{k+\ell_k}$ is positive and has the moment solution $\tilde\sigma$. Note
\begin{equation}
\int_{-\infty}^{\infty} x^k \mathrm{d}\tilde\sigma(x)=\int_{-\infty}^{\infty} x^{k+\ell_k} \mathrm{d} \sigma(x).
\end{equation}
Then for any polynomial ${\displaystyle P(x)=\sum_{k=0}^n a_k x^k}$, we have
\begin{equation}
\int_{-\infty}^{\infty} P(x) \mathrm{d} \tilde\sigma(x)
=\int_{-\infty}^{\infty} P_{\ell_k}(x) \mathrm{d}\sigma(x),
\end{equation}
where ${\displaystyle P_{\ell_k}(x)=\sum_{k=0}^n a_k x^{k+\ell_k}}$.
\end{ex}

\begin{thm}\label{sigmatilde}
Let $\s$ be a positive sequence and $\st$ be its subsequence given by $\tilde s_k=s_{k+\ell_0}$ for a fixed $\ell_0 \in 2\N_0$. Let $\sigma$ and $\tilde\sigma$ be the moment solutions of the sequences $\s$ and $\st$, respectively. If one of the moment problems is determinate, then
\begin{equation}\label{eqnsols}
\int_{-\infty}^{\infty} f(x) x^{\ell_0} \mathrm{d}\sigma(x)
= \int_{-\infty}^{\infty} f(x) \mathrm{d} \tilde{\sigma}(x)
\quad \textup{for all }   f\in \mathrm{L}^2_{\sigma}\cap \mathrm{L}^2_{\tilde\sigma}.
\end{equation}
\end{thm}
\begin{proof} 
To prove the first assertion, define the following two linear functionals:
$$\Phi_1(f(x))
=\int_{-\infty}^{\infty} f(x) x^{\ell} \mathrm{d} \sigma(x)
\quad \textup{and }
\Phi_2(g(x))
=\int_{-\infty}^{\infty}  g(x) \mathrm{d} \tilde{\sigma}(x)$$
for all $f\in \mathrm{L}^2_{\sigma}$ and $g\in \mathrm{L}^2_{\tilde\sigma}$.
To see $\Phi_1$ and $\Phi_2$ are bounded, we have
$$\Big|\Phi_1(f(x))\Big|
\leq \int_{-\infty}^{\infty} \left|f(x) x^{\ell_0} \right| \mathrm{d} \sigma(x) 
\leq \|f(x)\|_{\mathrm{L}^2_{\sigma}}\sqrt{s_{2\ell_0}}<\infty$$
and
$$\Big| \Phi_2(f(x)) \Big| 
\leq\int_{-\infty}^{\infty} \left|g(x)\right| \mathrm{d} \tilde{\sigma}(x)
\leq \|g(x)\|_{\mathrm{L}^2_{\tilde\sigma}}\sqrt{s_{0}}<\infty.$$
Now for any function $f\in \mathrm{L}^2_{\sigma}\cap \mathrm{L}^2_{\tilde\sigma}$, $$\Phi(f(x))=\Phi_1(f(x))-\Phi_2(f(x))$$ 
is a bounded linear functional.\\
Observe that for any polynomial $P(x)=a_n x^n + \cdots + a_1 x + a_0$,
\begin{align*}
\Phi(P(x))
= & \int_{-\infty}^{\infty} P(x)x^{\ell_0} \mathrm{d} \sigma(x) 
 -\int_{-\infty}^{\infty} P(x) \mathrm{d} \tilde{\sigma}(x)\\
= &\int_{-\infty}^{\infty}  \left[a_n x^{n+\ell_0} + \cdots + a_0 x^{\ell_0}\right] \mathrm{d} \sigma(x)
-\int_{-\infty}^{\infty} \left[ a_n x^n + \cdots + a_0\right] \mathrm{d} \tilde{\sigma}(x)\\
= & a_n(s_{n+\ell_0} - \tilde s_n) + \cdots + a_1(s_{1+\ell_0}
-\tilde s_1)+a_0(s_{\ell_0}-\tilde s_0)\\
= & 0.
\end{align*}
Due to the given determinacy of the moment problems, the set of polynomials is dense in $\mathrm{L}^2_{\sigma}$ or $\mathrm{L}^2_{\tilde\sigma}$. Hence the set of polynomials is dense in $\mathrm{L}^2_{\sigma}\cap \mathrm{L}^2_{\tilde\sigma}$. Therefore, by the Hahn-Banach Theorem, $\Phi$ is identically zero. Thus equation \eqref{eqnsols} holds.
\end{proof}

The above theorem is particularly useful when the first finitely many moments are missing. It says that we can find the moment solution with the remaining data. Whether the moment solution of the subsequence is determinate or not is to be studied more carefully in comparison with the determinacy of the moment solution of the original sequence. It is, however, known that the first term in the subsequence $\{s_{k+\ell_0}\}_{k\in \N_0}$ can be modified to make the moment problem determinate, but the solution will not remain the same.

The following theorem gives us a relation between $\mathrm{L}^2$ spaces related to a moment sequence and its submoment sequence. Here we require an evenness assumption on $\ell_0$ to guarantee that $\nu$ is a nondecreasing measure.

\begin{cor}
Consider a positive sequence $\s$ and its subsequence $\st$ given by $\tilde s_k=s_{k+\ell_0}$ for some $\ell_0\in 2\N_0$. Let $\sigma$ and $\tilde\sigma$ be the moment solutions of the sequence and the subsequence, respectively. Then there is a measure $\nu$, absolutely continuous with respect to $\sigma$, such that
\begin{equation}\label{inclusion}
\left(\mathrm{L}^2_{\sigma}\cap \mathrm{L}^2_{\tilde\sigma}\right)\subseteq \mathrm{L}^2_{\nu}.
\end{equation}
\end{cor}
\begin{proof}
Let $f\in \mathrm{L}^2_{\sigma}\cap \mathrm{L}^2_{\tilde\sigma}$. By equation \eqref{eqnsols} we have
$$\int_{-\infty}^{\infty} f(x) x^{\ell_0} \mathrm{d} \sigma(x)
=\int_{-\infty}^{\infty} f(x) \mathrm{d} \tilde{\sigma}(x)
\quad \textup{for all }
f\in \mathrm{L}^2_{\sigma}\cap \mathrm{L}^2_{\tilde\sigma}.$$
Set $\mathrm{d}\nu = x^{\ell_0} \mathrm{d}\sigma$. Then $ \nu << \sigma $ and
\begin{equation}
\int_{-\infty}^{\infty} f(x) \mathrm{d} \nu(x)
= \int_{-\infty}^{\infty} f(x) \mathrm{d} \tilde{\sigma}(x),
\end{equation}
and $f\in \mathrm{L}^2_{\nu}.$
\end{proof}

Now consider a general submoment sequence $\st$ given by $\tilde s_k=s_{k+\ell_k}$ for an appropriate $\{\ell_k\} \subseteq \N_0$. It was shown in the previous section that for any $k\in \N_0$, $$\ell_{k}=kd+\ell_0, \text{ where } d \in \N_0\text{ and }\ell_0\in 2\N_0,$$ $ \st $ is a positive sequence. Modifying the functional $\Phi_1$ in the proof of Theorem \ref{sigmatilde} gives the following result.

\begin{thm}\label{sigmatildegeneral}
Let $\s$ be a positive sequence and $\st$ be its subsequence given by $\tilde s_k=s_{kd+\ell_0}$. Let $\sigma$ and $\tilde\sigma$ respectively be the moment solutions of the sequences $\s$ and $\st$. If one of the moment problems is determinate, and $f(x^d)\in \mathrm{L}^2_{\sigma}$ for any $f\in \mathrm{L}^2_{\sigma}$, then
\begin{equation}\label{eqnsols2}
\int_{-\infty}^{\infty} f(x^d) x^{\ell_0} \mathrm{d} \sigma(x)
=\int_{-\infty}^{\infty} f(x) \mathrm{d} \tilde{\sigma}(x) \quad \textup{for all } f\in \mathrm{L}^2_{\sigma}\cap \mathrm{L}^2_{\tilde\sigma}.
\end{equation}
\end{thm}

A precise relation between $\mathrm{L}^2$ spaces of moment solutions of the original sequence and that of one of its positive subsequences can be useful in characterizing the submoment solutions and hence approximating the missing data.

Exploring equation \eqref{eqnsols2} a little further, for $\lambda\in \C$ with $y=
\operatorname{Im} \lambda\neq 0$, 
and define $$f(x)=\frac{1}{x-\lambda}.$$
Since $|x-\lambda|\geq|y|$,
$$\int_{-\infty}^{\infty} |f(x)|^2 \mathrm{d} \sigma(x)
= \int_{-\infty}^{\infty} \frac{\mathrm{d}\sigma(x)}{|x-\lambda|^2}
\leq \int_{-\infty}^{\infty} \frac{\mathrm{d}\sigma(x)}{|y|^2}=\frac{s_0}{|y|^2}<\infty.$$
 
Therefore, $f\in \mathrm{L}^2_{\sigma}$. By the same argument, $f\in \mathrm{L}^2_{\tilde\sigma}$. Similarly, it can be shown that $f(x^d)\in \mathrm{L}^2_{\sigma}$. Then by Theorem \ref{sigmatildegeneral},
\begin{equation}\label{quotient}
\int_{-\infty}^{\infty}\frac{x^{\ell_0} \mathrm{d}\sigma(x)}{x^d-\lambda}
=\int_{-\infty}^{\infty} \frac{\mathrm{d}\tilde\sigma(x)}{x-\lambda}.
\end{equation}
Applying Nevanlinna's theorem \cite{akhiezer} to equation \eqref{quotient} yields the following result which connects a moment problem with the polynomials corresponding to its submoment problem.

\begin{thm}\label{subnevanlinna}
Let $\s$ be a positive sequence and $\st$ be its subsequence given by $\tilde s_k=s_{kd+\ell_0}$. Let $\sigma$ and $\tilde\sigma$ respectively be the moment solutions of the sequences $\s$ and $\st$. If the moment problem of $\st$ is indeterminate, then
\begin{equation}\label{eqnsubnevanlinna}
\int_{-\infty}^{\infty}\frac{x^{\ell_0} \mathrm{d}\sigma(x)}{x^d-\lambda}
=-\frac{\tilde A(\lambda)\phi(\lambda)-\tilde C(\lambda)}{\tilde B(\lambda)\phi(\lambda)-\tilde D(\lambda)},
\end{equation}
where $\tilde A(\lambda), \tilde B(\lambda), \tilde C(\lambda), \tilde D(\lambda)$ form a Nevanlinna matrix of the submoment problem, and $\phi\in N.$
\end{thm}

For the shifted subsequence in Theorem \ref{sigmatilde}, we want to investigate determinacy of its moment problem of a special submoment sequence in relation to determinacy of the original moment problem. Recall that the limit circles $K_{\infty}(\lambda)$ ($\lambda \in \C$ and $\operatorname{Im}\lambda\neq0$) provide a fundamental concept for studying determinacy of moment problems.

\begin{thm}\label{circpoints123}
For the same moment sequences and conditions as in Theorem \ref{sigmatilde}, for $\lambda \in \C, \operatorname{Im}\lambda\neq 0$ the following holds.
\begin{equation}\label{circlepoints}
\int_{-\infty}^{\infty} \frac{\mathrm{d} \tilde\sigma(x)}{x-\lambda}
= C+\lambda^{\ell_0}\int_{-\infty}^{\infty} \frac{\mathrm{d}\sigma(x)}{x-\lambda},
\end{equation}
where $C$ is a constant depending on $\lambda$.
\end{thm}

Let $$w_{\sigma}(\lambda)
= \int_{-\infty}^{\infty} \frac{\mathrm{d}\sigma(x)}{x-\lambda}
\quad \text{and  } 
w_{\tilde\sigma}(\lambda)
= \int_{-\infty}^{\infty}\frac{\mathrm{d}\tilde{\sigma}(x)}{x-\lambda}.$$
Then Theorem \ref{circpoints123} states that
\begin{equation}\label{circrelation}
w_{\tilde\sigma}(\lambda)
= C+\lambda^{\ell_0}w_{\sigma}(\lambda)
\end{equation}
for $ \s$ and $ \st $ specified in that theorem. Recall that points $w_{\sigma}$ and $w_{\tilde\sigma}$ lie on the circumferences of the circles $K_{\infty}(\lambda)$ and $\tilde K_{\infty}(\lambda)$ corresponding to the moment problems of $\s$ and $\st$ respectively. Equation \eqref{circrelation} precisely describes the distortion of the circle corresponding to the original moment sequence in relation to its tail.

\section{Perturbation or Modification of Moment Sequences}

Assume that $s_0>|s_k|$ for all $k\in \N$. Then we can have a wide range of choices for $s_0$ keeping the rest of terms constant and maintaining positivity. An important application of modification in $s_0$ is that it can lead to a determinate solution. It was proved by Stieltjes in his 1894-Memoir \cite{stieltjes} that a determinate moment solution can be obtained from an indeterminate one with a modification in $s_0$. The moment sequence $s_n=q^{\frac{-(n+1)^2}{2}}$ gives the Stieltjes-Wigert polynomials $P_n(x;q)$, which are orthogonal in a log-normal distribution, known to be indeterminate.
The modified sequence $\st$ defined as $$\tilde s_0=s_0-\frac{1}{\sum_{n=0}^{\infty}[P_n(0;q)]^2}$$
and $\tilde s_n=s_n$ for all $n\geq 1$ has a determinate moment solution $\tilde\sigma(x)$ given by 
$$\tilde\sigma(x)=\sum_{x\in U}c_x\delta_x,$$
where $U$ is the zero set of the reproducing kernel 
$$K(0,w)=\sum_{n=0}^{\infty}P_n(0)P_n(w)$$ 
and 
$$c_x=\frac{1}{\sum_{k=0}^{\infty}[P_k(x;q)]^2}, \ \ x\in U.$$ 
An example of an application of modified moments to harmonic solids is given in \cite{blumstein}. 

To study the stability of perturbation of moment sequences with arbitrary sequences, the following well known lemma is useful \cite{book:widder}.
\begin{lem}\label{lem:signedmeasure}
Let $ \ts $ be an arbitrary sequence of real numbers. Then there exists a signed measure $ \mu $ of bounded variation such that $$ \int_{-\infty}^\infty x^k d\mu(x) = t_k , \ \ k=0,1,2,\cdots$$
\end{lem}
Let $ \s $ be a positive sequence, and $ \sigma $ be a corresponding positive measure generating it.  By the Hahn-Jordan theorem, since a signed measure of bounded variation $ \mu $ can be decomposed into a difference of two nonnegative measures $ \mu_1 $ and $ \mu_2 $, $ \mu= \mu_1 - \mu_2 $, perturbation of the moment sequence $ \s $ by the sequence $ \ts $ is a moment sequence if and only if $ \sigma - \mu_2 $ is a  positive measure, i.e.  $ \mu_2 (E) \leq \sigma (E) $ for every $ \sigma $-measurable set $ E $.  Since this readily implies that $ \mu_2 $ is absolutely continuous with respect to $ \sigma $, letting $ f \in \mathrm{L}^1 (d\sigma) $ denote the Radon-Nikodym derivative of $ \mu_2 $ with respect to $ \sigma $, the above holds if and only if $ |f| \leq 1 $ $ \sigma$-a.e. Since $ \mu_2 $ is also a nonnegative measure, $ 0 \leq f \leq 1 $. We shall say that a signed measure $ \mu $ is {\it dominated} by $ \sigma $, if there exists an $ f \in L^1(\sigma) $, $ 0 \leq f \leq 1 $, such that the negative part of $ \mu $, $ \mu_2 $, satisfies $ d\mu_2 = f d\sigma $.
\begin{thm}\label{perturbation}
Let $ \s $ be a moment sequence and $ \ts $ be an arbitrary sequence for which the signed measure generating it is dominated by $ \sigma $.  Then 
$\{s_k+t_k\}_{k \in \N0} $ is a moment sequence.
\end{thm}
Clearly if $ \mu $ is dominated by $ \sigma $, then $ \varepsilon \mu $ is also dominated by $ \sigma $, for all $ 0 \leq \varepsilon \leq 1 $. 

While Theorem \ref{perturbation} gives a condition for the perturbation of a moment sequence $ \s $ by an arbitrary sequence $ \ts $ so that $ \{s_k + \varepsilon t_k\}_{k \in \N_0} $ is a moment sequence for $ 0 \leq \varepsilon \leq 1 $, this condition is in terms of the signed measure generating the sequence 
$\ts$ and can be quite difficult to verify in practice.  

For a truncated moment sequence, perturbations prescribed by Theorem \ref{perturbation} are relatively easy to describe. Recall that by the Tchakaloff's theorem \cite{putinar-PAMS}, if $ \s $ is a truncated moment sequence, then there exist $ \{p_1, \cdots , p_m\} \subset \R $ and $ c_i > 0 $, such that 
\begin{equation}\label{pointmass}
\sigma = \sum_{i=1}^m c_i \delta_{p_i} ,
\end{equation}
where $ \delta_{p_i} $ is an unit point masse at $ p_i $. Then it holds the following statement.

\begin{cor}\label{truncatedperturb}
Let $ \s $ be a truncated moment sequence generated by $ \sigma $ with the form \eqref{pointmass} and $ \ts $ be an arbitrary real sequence generated by $ \mu= \mu_1 - \mu_2 $.  Then $ \{s_k+t_k\}_{k \in \N_0 } $ is a moment sequence if and only if the support of $ \mu_2 $ is in the support of $ \sigma $ and if $ d_j \geq 0 $ is the weight of the point masses of $ \mu_2 $ at $ p_j $, then $ d_j \leq c_j $.
\end{cor}
\begin{proof}
Let $ \sigma $ be an atomic measure as specified by Tchakaloff's theorem. For $ \sigma $ to dominate $ \mu_2 $, since $ \mu_2 << \sigma $, the support of $ \mu_2 $ is a subset of the support of $ \sigma $. For $ \sigma - \mu_2 $ to be a positive measure, we must have the weight of the point masses of $ \mu_2 $ at $ p_j $ must be less than $ c_j $. Hence the necessary part follows.
The converse is obvious. 
\end{proof}
 

For a non-truncated moment sequence $ \s $, letting $ \{ t_{1,k} \}_{k \in \N_0} $ and $  \{t_{2,k}\}_{k \in \N_0} $ be the two positive sequences generated by the positive and negative parts of the signed measure $ \mu $ in Lemma \ref{lem:signedmeasure}, $ t_k = t_{1,k} - t_{2,k} $. Then
$$
t_{2, 2k} = \int_{-\infty}^{\infty} x^{2k} d \mu_2 (x) = \int_{-\infty}^\infty x^{2k} f(x) d\sigma (x) \leq s_{2k}.$$ 
However, a simple condition on 
$ t_{2, 2k+1} $ to guarantee that $ \{s_k - t_{2,k} \}_{k \in \N_0} $ is a positive sequence for an arbitrary positive sequence $ \s $ eludes us at this time. 




\begin{center}
{\bf References}
\end{center}

\bibliographystyle{plain}
\bibliography{refs}

\end{document}